\documentclass[english, makeidx, 12]{amsart}

\usepackage{multicol} 
\usepackage{babel}
\usepackage{amstext}
\usepackage{amsmath}
\usepackage{amsfonts}
\usepackage{latexsym}
\usepackage{ifthen}
\usepackage{xypic}
\usepackage{amssymb}

\xyoption{all}
\newcommand{\Rien}[1]{}
\newcommand{\Formel}[1]{(\ref{#1})}

\newcommand{\BN}{\mathbb B}

\newcommand{\Prop}[1]{Proposition~\ref{#1}}
\newcommand{\Cor}[1]{Corollary~\ref{#1}}

\newcommand{\Theo}[1]{Theorem~\ref{#1}}

\renewcommand{\O}{{\mathcal O}}

\newcommand{\PN}{{\mathbb P}}

\newcommand{\Pic}{{\rm Pic}}

\newcommand{\lra}{\longrightarrow}
\newcommand{\KC}{{\mathbb C}}
\newcommand{\KZ}{{\mathbb Z}}
\newcommand{\KQ}{{\mathbb Q}}

\newcommand{\KN}{{\mathbb N}}
\newcommand{\End}{{\rm End}}

\newcommand{\Sieg}{\mathfrak H}
\newcommand{\OH}{\mathbb H}
\newcommand{\KR}{\mathbb R}
\newcommand{\sO}{{\mathcal O}}

\newtheorem{lemma1}[equation]{}

\newenvironment{lemma}{\begin{lemma1}{\bf Lemma.}}{\end{lemma1}}

\newenvironment{example}{\begin{lemma1}{\bf Example.}\rm}{\end{lemma1}}

\newenvironment{theorem}{\begin{lemma1}{\bf Theorem.}}{\end{lemma1}}

\newenvironment{proposition}{\begin{lemma1}{\bf Proposition.}}{\end{lemma1}}
\newenvironment{corollary}{\begin{lemma1}{\bf Corollary.}}{\end{lemma1}}

\newenvironment{rem}{\begin{lemma1}{\bf Remark.}\rm}{\end{lemma1}}
\newenvironment{definition}{\begin{lemma1}{\bf Definition.}}{\end{lemma1}}

\begin{document}

\title{Holomorphic normal projective connections on projective manifolds}
\author[P. Jahnke]{Priska Jahnke}
\address{Priska Jahnke - Institut f\"ur Mathematik- Freie Universit\"at Berlin
  - Arnimallee 3 - D-14195 Berlin, Germany}
\email{priska.jahnke@fu-berlin.de}
\author[I.Radloff]{Ivo Radloff}
\address{Ivo Radloff - Mathematisches Institut - Universit\"at Bayreuth -
  D-95440 Bayreuth, Germany}
\email{ivo.radloff@uni-bayreuth.de}
\date{\today}
\maketitle

\section*{Introduction}
 The notion of projective connections traces back to E.Cartan. The existence of a {\em holomorphic} projective connection on a compact complex manifold $M$ is a strong condition as soon as $\dim M > 1$. A result of Kobayashi and Ochiai (\cite{KO}, \Theo{KE}) says that a compact K\"ahler Einstein (KE) manifold $M_m$ admits a holomorphic projective connection if and only if $M$ is
\begin{enumerate}
  \item $\PN_m(\KC)$
  \item a finite {\'e}tale quotient of a torus or
  \item a ball quotient, i.e., the universal covering space of $M$ can be identified with the open unit ball $\BN_m(\KC)$ in $\KC^m$.
\end{enumerate}
Equivalently, the holomorphic sectional curvature is constant on $M$ or, in purely numerical terms, the Chen--Ogiue inequality is an equality:
           \[2(m+1)c_2(M) = mc_1(M)^2.\]
Kobayashi and Ochiai also gave a classification of all compact complex surfaces carrying such a connection (\cite{KO}). For K\"ahler surfaces their result says that any such surface is KE.

Before the authors' previous work in \cite{JRproj} it was not known whether there exist non--KE examples of compact K\"ahler manifolds carrying such a connection. In \cite{JRproj} the complete classification of projective threefolds with a holomorphic projective connection was given with one additional example: modular families of fake elliptic curves over a Shimura curve.

In this article the classification is extended to projective manifolds of arbitrary dimension -- up to the question of abundance (see explanations below):

\begin{theorem} \label{main} Let $M$ be an $m$--dimensional projective manifold with a holomorphic projective connection $\Pi$. Assume that $M$ is {\em not} K\"ahler Einstein, i.e., not one of the examples above. Then $K_M$ is nef. If $K_M$ is even abundant and if $\Pi$ is flat, then, up to a finite {\'e}tale covering, $M$ is an abelian group over a compact Shimura curve $C$, i.e., there exists a holomorphic submersion 
    \[f: M \lra C\]
such that every fiber is smooth abelian and such that $f$ has a smooth section. Moreover, the following equivalent conditions are satisfied:
\begin{enumerate}
 \item the Arakelov inequality (\cite{Fa})
  \[2\deg f_*\Omega^1_{M/C} \le (m-1)\deg K_C = 2g_C-2\]
 is an equality. 
 \item $M \simeq Z \times_C Z \times_C \cdots \times_C Z$ where $Z \lra C$ a Kuga fiber space constructed from the rational corestriction $Cor_{F/\KQ}(A)$ of a division quaternion algebra $A$ defined over a totally real number field $F$ such that
   \[A \otimes_{\KQ} \KR \simeq M_2(\KR) \oplus {\mathbb H} \oplus \cdots \oplus {\mathbb H}.\]
Here ${\mathbb H}$ denotes the Hamiltonian quaternions (see section~\ref{Ex}).
\end{enumerate}
Conversely, any abelian scheme as in 1.) or 2.) admits a flat projective connection.
\end{theorem}
The equivalence of 1.) and 2.) for abelian group schemes over curves is a result of Viehweg and Zuo (\cite{VZ}). The construction of the family $Z \lra C$ in 2.) is a generalization of a construction of modular families of abelian varieties due to Mumford (\cite{Mum}). For the proof of \Theo{main} see section~\ref{klassalles}. Note that \Theo{main} gives a result analogous to the KE case: an explicit method of how to construct the examples as well as a description in purely numerical terms. Some explanations:

The canonical divisor $K_M = \det\Omega_M^1$ is called nef if $K_M.C \ge 0$ for any irreducible curve $C$ in $M$. It is called abundant if $|dK_M|$ is spanned for some $d \gg 0$ defining a holomorphic map $f: M \lra Y$ onto some (perhaps singular) projective variety $Y$. The map is called {\em Iitaka fibration} of $M$.

{\em Abundance Conjecture (\cite{KM}).} Let $M$ be a projective manifold. Does $K_M$ nef imply $K_M$ abundant? 

The conjecture is known to hold true in dimensions $\le 3$ and in many particular cases. In different terms, \Theo{main} gives a complete classification of projective manifolds carrying a flat holomorphic projective structure in any dimension in which Abundance Conjecture is true.

The construction of the examples is explained in detail in section~\ref{Ex}, we only give some details here. Let $F$ be a totally real number field, $F:\KQ = d$. Let $A/F$ be as in 2.) of \Theo{main}. The rational corestriction $Cor_{F/\KQ}(A)$ is a central simple $\KQ$--algebra of dimension $4^d$ and either
  \[Cor_{F/\KQ}(A) \simeq M_{2^d}(\KQ) \quad \mbox{or} \quad Cor_{F/\KQ}(A) \simeq M_{2^{d-1}}(B),\]
 for some division quaternion algebra $B/\KQ$. We refer to the first case as '$B$ splits'. From this setting one constructs (\cite{Mum}, \cite{VZ}) a modular family $f: Z \lra C$ of abelian varieties over some compact Shimura curve $C$ with simple general fiber $Z_{\tau}$ satisfying
\begin{itemize}
  \item $\dim Z_\tau = 2^{d-1}$ and $End_{\KQ}(Z_{\tau}) \simeq \KQ$ in the case $B$ split,
  \item $\dim Z_\tau = 2^d$ and $End_{\KQ}(Z_{\tau}) \simeq B$ in the case $B$ non--split.
\end{itemize}

The smallest possible dimension is obtained for $F = \KQ$, $A = B$ an indefinite division quaternion algebra. Here $f: Z \lra C$ is a (PEL--type) family of abelian surfaces, $End_{\KQ}(Z_{\tau}) \simeq B$ for $\tau$ general. Such abelian surfaces are called fake elliptic curves above, it is the example found in \cite{JRproj}. Since the abundance conjecture is true in dimension $\le 3$ we have the following corollary to \Theo{main}:
\begin{corollary}
  Let $M$ be a projective manifold with a flat holomorphic normal projective connection of dimension at most $3$ which is not K\"ahler Einstein. Then $\dim M = 3$ and $M$ is, up to {\'e}tale coverings, a modular family of fake elliptic curves.
\end{corollary}
The Corollary includes Kobayashi and Ochai's classification in the (flat, projective) surface case and the authors' previous result on projective threefolds.

\

\noindent {\bf Acknowledgements.} The authors want to thank N. Nakayama for valuable explanations concerning abelian fibrations.

\

\noindent {\bf Notations.} We consider only complex manifolds $M$; $T_M$ denotes the holomorphic tangent bundle, $\Omega_M^1$ the dual bundle of holomorphic one forms. We do not distinguish between line bundles and divisors on $M$. $K_M = \det \Omega_M^1$ denotes the canonical divisor.

\section{Holomorphic normal projective connections} \setcounter{equation}{0}

There are various definitions of projective structures and connections. We essentially follow Kobayashi
and Ochiai (\cite{KO}). 

\subsection{Holomorphic normal projective connections (H.n.p.c.)}
First recall the definition of the {\em Atiyah class} (\cite{At}). Associated to a holomorphic vector bundle $E$ on the complex manifold $M$ one has the first jet sequence
  \[0 \lra \Omega_M^1 \otimes E \lra J_1(E) \lra E \lra 0.\]
Obstruction to the holomorphic splitting is a class $b(E) \in H^1(M, Hom(E,E)\otimes \Omega_M^1)$ called {\em Atiyah class of $E$}. For properties of $b(E)$ see \cite{At}. We only mention that if $\Theta^{1,1}$ denotes the $(1,1)$--part of the curvature tensor of some differentiable connection on $E$, then, under the Dolbeault isomorphism, $b(E)$ corresponds to $[\Theta^{1,1}] \in H^{1,1}(M, Hom(E,E))$. In particular, for $M$ K\"ahler, $tr(b(E)) = -2i\pi c_1(E) \in H^1(M, \Omega_M^1)$. This is why we normalise and put $a(E) := -\frac{1}{2i\pi}b(E)$.

\begin{definition}
$M_m$ carries a holomorphic normal projective connection (h.n.p.c.) if the (normalised) Atiyah class of the holomorphic cotangent bundle
has the form
  \begin{equation} \label{AtProj} 
a(\Omega_M^1) = \frac{c_1(K_M)}{m+1} \otimes id_{\Omega_M^1} + id_{\Omega_M^1}
  \otimes \frac{c_1(K_M)}{m+1} \in H^1(M, \Omega_M^1 \otimes T_M \otimes
  \Omega_M^1).
  \end{equation}
(Note $\Omega_M^1 \otimes T_M \otimes
\Omega_M^1 \simeq
\End(\Omega_M^1) \otimes \Omega_M^1 \simeq \Omega_M^1 \otimes
\End(\Omega_M^1)$.) 
\end{definition}
It was shown in \cite{MM} how a holomorphic cocycle solution to
\Formel{AtProj} can be thought of as a $\KC$--bilinear holomorphic connection map
  $\Pi: T_M \times T_M \to T_M$
satisfying certain rules modelled after the Schwarzian derivative. Conversely, the existence of such a connection implies \Formel{AtProj}.

Let $\tilde{M} \to M$ be finite {\'e}tale. Then $\tilde{M}$ carries a h.n.p.c. if and only if $M$ carries a h.n.p.c.

\subsection{Holomorphic projective structures} The manifold $M$ is said to admit a {\em holomorphic projective structure} or a {\em flat} h.n.p.c. if there exists an atlas $\{(U_i, \varphi_i)\}_{i \in I}$ with holomorphic maps $\varphi_i: U_i \hookrightarrow \PN_m(\KC)$ such that
  \[\varphi_i \circ \varphi_j^{-1}: \varphi_j(U_{ij}) \lra \varphi_i(U_{ij})\]
is the restriction of some $g_{ij} \in PGl_m(\KC)$ whenever $U_{ij} = U_i \cap U_j \not= \emptyset$. The following fact is fundamental: {\em If $M$ admits a holomorphic projective structure, then $M$ admits a h.n.p.c}. Whether the convers is true is not known. $M$ admits a holomorphic projective structure if and only if $\Pi = 0$ is a cocycle solution to \Formel{AtProj}.

Let $\tilde{M} \to M$ be {\'e}tale. Then $\tilde{M}$ and $M$ are locally isomorphic and $\tilde{M}$ carries a holomorphic projective structure if and only if $M$ does.

\begin{example} \label{Exmpl} Examples of projective or K\"ahler manifolds admitting a (flat) h.n.p.c. are:
 
1.) $\PN_m(\KC)$. 

2.) Any manifold $M = M_m$ whose universal covering space $\tilde{M}$ can be embedded into $\PN_m(\KC)$ such that $\pi_1(M)$ acts by restricting automorphisms from $PGl_{m+1}(\KC)$.

3.) Ball quotients, i.e., $M$ with $\tilde{M} = \BN_m(\KC)$, the non compact dual of $\PN_m(\KC)$ in the sense of hermitian symmetric spaces. This is a special case of 2.).

4.) Finite {\'e}tale quotients of tori.

\noindent There are more examples without the assumption $M$ projective/K\"ahler: 

5.) tori, certain Hopf manifolds, twistor spaces over conformally flat Riemannian
fourfolds.
\end{example}

\begin{rem}
  There is an analogous notion of an $S$--structures, $S$ an arbitrary hermitian symmetric space of the compact type (\cite{KObook}). The rank one case $S = \PN_m(\KC)$ is somewhat special in this context.
\end{rem}

\subsection{Development} (\cite{KO}, \cite{KoWu}) Let $M_m$ carry a flat h.n.p.c. and let $\{(U_i, \varphi_i)\}_{i\in I}$ be a projective atlas. 
Choose one coordinate chart $(U_0, \varphi_0)$. Given a point $p \in M_m$, choose a chain of charts $(U_i, \varphi_i)_{i=0, \dots, r}$ such that $U_i \cap U_{i-1} \not= \emptyset$, $i = 1, \dots, r$ and $p \in U_r$. We have
 \[\varphi_{i-1}\circ \varphi_i^{-1} = g_i|_{\varphi_i(U_i\cap U_{i-1})}\]
for projective transformations $g_i \in PGl_m(\KC)$. Set $\psi(p) := (g_1 \circ g_2 \circ \cdots \circ g_r\circ \varphi_r)(p) \in \PN_m(\KC)$. This gives a multivalued map $\psi: M \to \PN_m(\KC)$ which is single valued in the case $M$ simply connected. It is called a {development} of $M$ to $\PN_m(\KC)$. By construction, $\psi$ is locally an isomorphism.

If $M$ is not simply connected, then let $\tilde{M}$ be its universal covering space and let $\psi: \tilde{M} \to \PN_m(\KC)$ be a development of $\tilde{M}$. We get a natural map $\rho: \pi_1(M) \to PGl_m(\KC)$ such that $\psi(\gamma(p)) = \rho(\gamma)(\psi(p))$ for any $\gamma \in \pi_1(M)$ (\cite{KoWu}).

\begin{example} Two examples of neighborhoods of an elliptic curve:
  
1.) Let $U \subset \PN_2$ be an open neighborhood of an elliptic curve $E$ such that $U$ is a retract of $E$. Then $U$ inherits a holomorphic projective structure from $\PN_2$. Let $\mu: \tilde{U} \to U$ be the universal covering space of $U$. Then $\tilde{U}$ inherits a holomorphic projective atlas from $U$ such that $\mu$ is a development of $\tilde{U}$.

2.) Let $Z \lra C$ be a modular family of elliptic curves as in example~\ref{ellcurv}. Then the embedding
  \[\iota: \tilde{Z} = \KC \times \Sieg_1 \hookrightarrow \PN_2(\KC), \quad (z, \tau) \mapsto [z:\tau:1].\]
is a development of $\tilde{Z}$.
\end{example}

\subsection{Chern Classes}
Let $M_m$ be as above compact K\"ahler with a h.n.p.c. Similar to projectice space one has (\cite{KO}):
\begin{equation} \label{chern}
  c_r(M) = \frac{1}{(m+1)^r}{m+1 \choose r}c_1^r(M), \; r=0, \dots, m \; \quad \mbox{in }
H^r(M, \Omega_M^r).
 \end{equation}
In particular, $2(m+1)c_2(M) = mc_1^2(M)$.

\section{Birational Classification (flat case)} \setcounter{equation}{0}
In the K\"ahler Einstein case, $2(m+1)c_2(M) = mc_1^2(M)$ if and only if the holomorphic sectional curvature of $M_m$ is constant. This gives the following result due to Kobayashi and Ochiai (\cite{KO}) from the introduction:
\begin{theorem} \label{KE}
  Let $M$ be K\"ahler-Einstein with a h.n.p.c. Then
  \begin{enumerate}
   \item $M \simeq \PN_m(\KC)$ or
   \item $M$ is an {\'e}tale quotient of a torus or
   \item $M$ is a ball quotient.
  \end{enumerate}
\end{theorem}
Recall that K\"ahler-Einstein implies $\pm K_M$ ample or $K_M \equiv 0$. Conversely, by Aubin and Yau's proof of the Calabi conjecture, $K_M$ ample implies $M$ K\"ahler-Einstein.

Consider the projective case from now on, i.e., let $M_m$ be a projective manifold with a h.n.p.c., not necessarily flat. If $K_M$ is not nef, then $M$ contains a rational curve by the cone theorem (\cite{KM}) meaning that there exists a non--constant holomorphic map
  \[\nu: \PN_1(\KC) \lra M.\]
Then already $M \simeq \PN_m(\KC)$:

\begin{proposition} \label{ratcurve}
  Let $M_m$ be a projective manifold with a h.n.p.c. If $M$ contains a rational curve, then $M \simeq \PN_m(\KC)$.
\end{proposition}

\noindent We include here a proof in the case of a flat connection, for the general case see \cite{JRproj}.

\begin{proof} Let $\psi: \tilde{M} \lra \PN_m(\KC)$ be a development of the universal covering space $\mu: \tilde{M} \lra M$. We have an induced map $\tilde{\nu}: \PN_1(\KC) \lra \tilde{M}$ such that $\mu \tilde{\nu} = \nu$. Then $\nu^*T_M = (\psi \circ \tilde{\nu})^*T_{\PN_m(\KC)}$ is ample. This forces $M \simeq \PN_m(\KC)$ by Mori's proof of Hartshorne's conjecture (\cite{Mori}, in particular \cite{MP} I, Theorem 4.2.).
\end{proof}

\begin{rem}
  \Prop{ratcurve} holds mutatis mutandis for any projective manifold with a flat $S$--structure, $S$ some irreducible hermitian symmetric space of the compact type: the existence of a rational curve implies $M$ uniruled and $M \simeq S$ by \cite{HM}.  
\end{rem}

Thus if $M \not\simeq \PN_m(\KC)$, then $K_M$ is nef and $M$ does not contain any rational curve. This has strong consequences: any rational map
  \[\xymatrix{M' \ar@{..>}[r] & M}\]
from a manifold $M'$ must be holomorphic. Indeed, if we first had to blow up $M'$ in order to make it holomorphic, then we could find a rational curve in $M$.  A manifold with this property is sometimes called {\em strongly minimal} or {\em absolutely minimal} (\cite{MoriClass}, \S 9).

\

The following is taken from  (\cite{Ko}, 1.7.): $M$ is said to have {\em generically large fundamental group} if for any $p \in M$ general and for any irreducible complex subvariety $p \in W \subset M$ of positive dimension
  \[{\rm Im}[\pi_1(W_{norm}) \lra \pi_1(M)] \quad \mbox{is infinite,}\]
where $W_{norm}$ denotes the normalization of $W$. In our case:

\begin{proposition} \label{lfg}
  Any $M \not\simeq \PN_m(\KC)$ with a flat h.n.p.c. has generically large fundamental group.
\end{proposition}

\begin{proof} By \Prop{ratcurve},  $M \not\simeq \PN_m(\KC)$ implies $K_M$ nef.
Assume to the contrary ${\rm Im}[\pi_1(W_{norm}) \to \pi_1(M)]$ is finite for some $W$ as above. Denote the normalization map by $\nu: W_{norm} \to W$.

Let $C_{norm}$ be some general curve in $W_{norm}$, i.e., the intersection of $\dim W - 1$ general hyperplane sections. Think of $C_{norm}$ as the normalization of $C := \nu(C_{norm}) \subset W$. We have
  \[\pi_1(C_{norm}) \lra \pi_1(W_{norm}) \lra \pi_1(M).\]
Then ${\rm Im}[\pi_1(C_{norm}) \lra \pi_1(M)]$ is finite. The kernel of $\pi_1(C_{norm}) \lra \pi_1(M)$ induces a finite {\'e}tale covering $C' \lra C_{norm}$ from a compact Riemann surface $C'$, such that $\mu: C' \lra C$ factors over $\tilde{M}$, the universal covering space of $M$. Let $\psi: \tilde{M} \lra  \PN_m(\KC)$ be a development. Denote the induced map $C' \lra \tilde{M} \lra \PN_m(\KC)$ by $\psi_1$. Then $\mu^*T_M = \psi_1^*T_{\PN_m(\KC)}$ is ample, contradicting $K_M$ nef.
\end{proof}

Now assume $M$ is abundant, i.e., that some multiple of $K_M$ is in fact spanned, defining the Iitaka fibration
  \begin{equation} \label{Iitaka}
    f: M \lra Y
  \end{equation}
onto some normal projective variety $Y$ of dimension $\dim Y = \kappa(M)$. Since $M$ does not contain any rational curve, $f$ is equidimensional (\cite{Kaw}, Theorem 2). Let $F$ be a general (connected) fiber. Then $K_F = K_M + \det N_{F/M}$ is torsion in $\Pic(F)$ and there exists a finite {\'e}tale map $\tilde{F} \to F$ such that $K_{\tilde{F}}$ is trivial. By Beauville's decomposition theorem (\cite{Bo}, \cite{Be}) we may assume
   \[\tilde{F} \simeq A_y \times B_y\]
where $A_y$ is abelian and $B_y$ is simply connected. By \Prop{lfg} $B_y$ must be a point. Thus $M$ admits a fibration whose general fiber is covered by an abelian variety (see again \cite{JRproj} for the non-flat case).

\

We say that the Iitaka fibration defines an {\em abelian group scheme structure on $M$}, in the case $Y$ smooth, $f$ submersive, every fiber of $f$ is a smooth abelian variety, and $f$ admits a smooth section. The next proposition is mainly a consequence of a result of Koll\'ar:

\begin{proposition} \label{AbSc} Let $M$ be a projective manifold with a flat h.n.p.c. and assume $K_M$ is abundant. Then there exists a commuting diagram
  \[\xymatrix{M' \ar[r]\ar[d] & M \ar[d] \\
               N'\ar[r] & N}\]
with the folliwng property: (i) The vertical maps are the Iitaka fibrations of $M'$ and $M$, respectively. (ii) The map $M' \lra M$ is finite {\'e}tale. (iii) $M' \lra N'$ is an abelian group scheme.
\end{proposition}

\begin{proof}[of \Prop{AbSc}] The general fiber of the Iitaka fibration $f: M \to Y$ is an {\'e}tale quotient of an abelian variety (see the explanations before \Prop{AbSc}). By \cite{Ko}, 6.3.~Theorem there exists a finite {\'e}tale covering $M' \to M$ such that $f': M' \to Y'$ is birational to an abelian group scheme $\varphi: A \to S$. Here $M' \to Y'$ is the Stein factorization of $M' \to M \to Y$. Note that $M'$ is abundant and $M' \lra Y'$ is the Iitaka fibration of $M'$.

After replacing $M$ by $M'$ we may assume that $M \to Y$ is birational to $A \to S$. We obtain a digram
  \begin{equation} \label{diagr}
   \xymatrix{A \ar[d]^{\varphi} \ar@{..>}[r]^h & M \ar[d]^f \\
              S  \ar@{..>}[r]^g & Y}
  \end{equation}
in which $h$ and $g$ are birational. The map $h$ must be holomorphic, as $M$ is strongly minimal. The map $\varphi$ has a smooth section. Then $g$ mus be holomorphic, too. There exists a Zariski open subset $U \subset Y$ such that the two fibrations $f: f^{-1}(U) \to U$ and $\varphi: \varphi^{-1}(g^{-1}(U)) \to g^{-1}(U)$ are isomorphic. The claim is that then $Y$ is smooth and $M \to Y$ is a smooth abelian fibration with a smooth section. 

First note that by \Formel{diagr} every fiber of $f$ is irreducible and reduced, and $h$ is generically 1:1 on every fiber $A_s$ of $\varphi$ for $s \in S$ arbitrary. We have a generically injective map
  \begin{equation} \label{bdlmap}
   T_{A_s} \lra T_A|_{A_s} \lra h^*T_{M}|_{A_s}
  \end{equation} 
and it remains to show that it is a bundle map. Indeed, then every fiber of $f$ is smooth, implying the smoothness of $Y$. The existence of a smooth section of $f$ then follows from the existence of a section of $\varphi$ and the absence of rational curves in $M$.

By construction, $dK_M = f^*H$ for some $H \in \Pic(Y)$. Then $(h^*c_1(M))|_{A_s} = 0$ in $H^1(A_s, h^*\Omega^1_M|_{A_s})$. By \Formel{AtProj} the Atiyah class of $h^*T_{M}|_{A_s}$ vanishes. The tangent bundle $T_{A_s}$ is trivial. The next Lemma shows that \Formel{bdlmap} is a bundle map. 
\end{proof}

\begin{lemma}
  Let $A$ be a torus and $E$ a vector bundle on $A$ with vanishing Atiyah class $a(E) \in H^1(A, E^* \otimes E \otimes \Omega_A^1)$. Let $0 \not= s \in H^0(A, E)$ be a holomorphic section. Then $Z(s) = \{p \in A | s(p) = 0\} = \emptyset$. 
\end{lemma}

\begin{proof}
  Let $\PN(E)$ be the hyperplane bundle associated to $E$ and $\pi: \PN(E) \to A$ be the projection map. The vanishing of $a(E)$ implies the holomorphic splitting of the relative tangent bundle sequence
   \[0 \lra T_{\PN(E)/A} \lra T_{\PN(E)} \lra \pi^*T_A \lra 0\]
 (\cite{PePoSc}, 3.1.). Then $T_{\PN(E)} \simeq  T_{\PN(E)/A} \oplus \pi^*T_A$ and $A$ acts as a group on $\PN(E)$. For $a \in A$ let $t_a$ be the induced translation map on $A$ and let $t'_a$ be the induced automorphism of $\PN(E)$. We get a commutative diagram
 \begin{equation}
   \xymatrix{\PN(E) \ar[d] \ar[r]^{t'_a} & \PN(E) \ar[d] \\
              A  \ar[r]^{t_a} & A.}
  \end{equation}
Then ${t'_a}^*\O_{\PN(E)}(1) \simeq \O_{\PN(E)}(1) \otimes \pi^*L$ for some $L \in \Pic(A)$. The push forward to $A$ shows $t_a^*E \simeq E \otimes L$ (\cite{Ha}, III, 9.3.). Then $t_a^*\det E \simeq \det E \otimes L^{\otimes r}$ for $r = rk E$. As $\det E \in \Pic^0(A)$ is invariant under translations, $L^{\otimes r} \simeq \O_A$. Then $L$ is trivial as $A$ acts continously.

Then $t_a^*E \simeq E$ for any $a \in A$ implying that $A$ acts on $H^0(A, E)$. The action must be trivial. Then $s(p+a) = s(p)$ for any $a \in A$ implying $Z(s) = \emptyset$.
\end{proof}

\section{Torus fibrations}
\setcounter{equation}{0}
In this section $M_m$ denotes a compact K\"ahler manifold with a h.n.p.c. Assume the existence of a holomorphic submersion 
  \[f: M_m \lra N_n\]
such that every fiber is a torus. Assume that $f$ has a smooth holomorphic section. We have the exact sequence
   \begin{equation} \label{reltang}
    0 \lra f^*\Omega_N^1 \stackrel{df}{\lra} \Omega_M^1 \lra
    \Omega^1_{M/N} \lra 0
   \end{equation}
of holomorphic forms and
\[E = E^{1,0} = f_*\Omega_{M/N}^1\]
is a holomorphic rank $m-n$ vector bundle on $N$ such that $f^*E \simeq
\Omega_{M/N}^1$ via the canonical map $f^*f_*\Omega_{M/N}^1 \to
\Omega_{M/N}^1$.

\begin{proposition} \label{fam}
 In the above situation, $N$ admits a h.n.p.c., and
    \begin{equation} \label{AtE}
      a\big(E(-\frac{K_N}{n+1})\big) = 0 \quad
      \mbox{in } H^1(N, \End(E) \otimes
      \Omega_N^1),\end{equation}
  where $a$ denotes the normalised Atiyah class.
\end{proposition}

\begin{rem} \label{remat}
1.) The formula is in terms of classes, we do not assume the
existence of a theta characteristic on $N$. 

2.) In the case $N$ a compact Riemann surface, \Formel{AtE} implies $E \simeq U \otimes \theta$ for some flat bundle $U$ coming from a representation of $\pi_1(N)$ (\cite{At}) and some theta characteristic $\theta$. 
\end{rem}

\Prop{fam} will be proved below, we will first derive some consequences. The
trace of the (normalised) Atiyah class gives the first Chern class, hence
  \begin{equation} \label{relK}
    c_1(E) = \frac{m-n}{n+1}c_1(K_N) \quad \mbox{in } H^1(N,
    \Omega_N^1).
  \end{equation}
Let as usual $K_{M/N} = K_M - f^*K_N$. Then $K_{M/N} = \det \Omega_{M/N}^1 = f^*\det E$. We may rewrite \Formel{relK} as follows:
\begin{corollary}\label{ChernEq}
  In the situation of \Prop{fam} the following identities hold in  $H^1(M, \Omega_M^1)$:
    \[c_1(K_{M/N}) = \frac{m-n}{n+1}c_1(f^*K_N) \quad \mbox{and} \quad
    c_1(K_M) = \frac{m+1}{n+1} c_1(f^*K_N)\]
  In particular, $c_1(K_M)$ and $c_1(f^*K_N)$ are proportional.
\end{corollary}
The next corollary covers in particular degenerate cases when $M_m$ is a product of tori and $f$ a projection or when $M_m = N_n = \PN_m(\KC)$ and $f$ is isomorphism of projective space:

\begin{corollary}
\begin{enumerate}
 \item 
If $c_1(K_M) = 0$ or $c_1(K_N) = 0$, then $M$ and $N$ are {\'e}tale quotients of tori.
 \item If $K_M$ or $K_N$ are not nef, then $M \simeq N \simeq \PN_m(\KC)$.
\end{enumerate}
\end{corollary}

\begin{proof}
  1.) By \Cor{ChernEq}, if $c_1(M) = 0$ or $c_1(N)=0$, then $c_1(K_M) = c_1(K_N) =
0$. By \Formel{chern} all Chern classes of $M$ and $N$ vanish. Then
$M$ and $N$ are K\"ahler-Einstein and covered by tori by \Theo{KE}.
2.) If $K_M$ or $K_N$ are not nef, then $K_M$ and $K_N$ are not
nef by \Cor{ChernEq}, and $M \simeq \PN_m(\KC)$ and $N \simeq \PN_n(\KC)$ by \Prop{ratcurve}. As $n>0$ by assumption, $m = n$ and
$f$ is an automorphism of projective space.
\end{proof}

\begin{proof}[Proof of \Prop{fam}]
  Denote the section of $f$ by $s: N \to M$. Consider
  the pull back to $N$ by $s$ of \Formel{reltang}
   \begin{equation} \label{reltang2}
    0 \lra \Omega_N^1 \stackrel{s^*df}{\lra} s^*\Omega_M^1 \lra
    s^*\Omega^1_{M/N} \simeq E \lra 0.
   \end{equation}
  We have the map $ds: s^*\Omega_M^1 \lra \Omega^1_N$.   
 As $(ds)(s^*df) = d(f \circ s) = id_{\Omega_N^1}$, we see that \Formel{reltang2}
 splits holomorphically. 

The normalised Atiyah class of
  $s^*\Omega_M^1$ is obtained from the one of $\Omega^1_M$ by
  applying $ds$ to the last $\Omega_M^1$ factor in
  \Formel{AtProj}. What we get is 
    \begin{equation} \label{resat} 
    a(s^*\Omega_M^1) = \frac{s^*c_1(K_M)}{m+1} \otimes ds + id_{s^*\Omega_M^1} \otimes
    \frac{c_1(s^*K_M)}{m+1}
\mbox{ in } H^1(N, s^*\Omega_M^1 \otimes s^*T_M \otimes \Omega_N^1),
\end{equation} where
 we carefully distinguish between $s^*c_1(K_M) \in H^1(N,
 s^*\Omega_M^1)$ and the class $c_1(s^*K_M) = ds(c_1(K_M)) \in H^1(N, \Omega_N^1)$.

 The Atiyah class of a direct sum is the direct sum of the Atiyah
  classes (\cite{At}). As the pull back of \Formel{reltang} splits
  holomorphically, we get the Atiyah classes of $\Omega_N^1$ and $E$ by
  projecting \Formel{resat} onto the corresponding summands. 

We first compute $a(E)$. The class $c_1(K_M) \in H^1(M,
\Omega_M^1)$ is the pull back of some class in $H^1(N,
\Omega_N^1)$; it therefore vanishes under $H^1$ of $\Omega_M^1 \to
\Omega^1_{M/N}$. This means the first summand in \Formel{resat}
vanishes if we project, while the second summand becomes
  \begin{equation} \label{AtEeins}
    id_E \otimes \frac{c_1(s^*K_M)}{m+1}  \;\; \mbox{in } H^1(N, E
  \otimes E^* \otimes
    \Omega_N^1)\end{equation}
and this is $a(E)$. The trace is
  \[c_1(E) = rk(E)  \frac{c_1(s^*K_M)}{m+1} \in H^1(N,
  \Omega_N^1)\]
The determinant of
\Formel{reltang2} gives the following identities of classes in $H^1(N, \Omega_N^1)$:
  \begin{equation} \label{AdFor}
     c_1(K_N) = c_1(s^*K_M) - c_1(E) =
     \frac{m+1-(m-n)}{m+1}c_1(s^*K_M)
  \end{equation}
Now \Formel{AtE} follows from \Formel{AtEeins}.

Next we compute $a(\Omega_N^1)$. We have to apply $ds$ to the
first factor of the first summand in \Formel{resat}. This gives
$\frac{c_1(s^*K_M)}{m+1}$. As the splitting maps give the identity we get
  \[\frac{c_1(s^*K_M)}{m+1} \otimes id_{\Omega_N^1} + id_{\Omega_N^1} \otimes
    \frac{c_1(s^*K_M)}{m+1}  \in H^1(N, \Omega_N^1 \otimes T_N \otimes
    \Omega_N^1)\]
and this is $a(\Omega_N^1)$. As we just saw in \Formel{AdFor}
  \[\frac{c_1(K_N)}{n+1} = \frac{c_1(s^*K_M)}{m+1}.\]
Replacing this in the above formula we see that $N$ admits a h.n.p.c. as in \Formel{AtProj}. The Proposition is proved.
\end{proof}

\begin{proposition} \label{ballq}
  In the situation of \Prop{fam}, assume that $f$ is the Iitaka fibration of $M$. Then $N$ is a ball quotient.
\end{proposition}

\begin{proof}
If $f$ is the Iitaka fibration defined by $|dK_M|$ for some $d \in \KN$, then $dK_M = f^*H$ for some ample $H \in \Pic(N)$. By \Cor{ChernEq} $c_1(f^*K_N) = \frac{n+1}{d(m+1)}c_1(f^*H)$. The restriction to the section shows that $K_N$ is ample. As $N$ admits a h.n.p.c., $N$ is a ball quotient by \Theo{KE}.
\end{proof}

\begin{example} The above ideas also apply to the following situation: let $f: M \lra N$ be a vector bundle over some manifold $N$. Then $\Omega^1_{M/N} \simeq f^*M^*$. Assume that the total space $M$ carries a h.n.p.c.. As in \Prop{fam}
\[a\big(M^*(-\frac{K_N}{n+1})\big) = 0 \quad
      \mbox{in } H^1(N, \End(M^*) \otimes
      \Omega_N^1).\]
For example let $M \subset \PN_m(\KC)$ be the complement of a linear subspace $\simeq \PN_n(\KC)$ of $\PN_m(\KC)$. $M$ inherits a projective structure from $\PN_m(\KC)$. Projection to $\PN_n(\KC)$ shows  
  \begin{equation}\label{UVec}
    M \simeq \sO_{\PN_n}(1)^{\oplus m-n},
  \end{equation}
i.e., $M$ carries the structure of a vector bundle over $\PN_n(\KC)$. Here $M^*(-\frac{K_N}{n+1}) \simeq \sO_{\PN_n}^{\oplus m-n}$ and the Atiyah class of this bundle vanishes. 
\end{example}

\setcounter{equation}{0}

\section{Abelian schemes} \label{klassalles}
Let $f: M_m \lra N_n$ be an abelian scheme, $E = E^{1,0} = f_*\Omega_{M/N}^1$ as in the preceeding section. In the case where $N$ is a compact Riemann surface of genus $g > 0$ the Arakelov inequality says
  \[2\deg E \le (m-1)\deg(K_N) = (m-1)(2g_N-2).\]
From \Cor{ChernEq} and results of Viehweg and Zuo we obtain:

\begin{proposition} \label{VZklass}
  Let $f: M \lra N$ be an abelian scheme. Assume that $M$ admits a h.n.p.c. and that $N$ is of general type, $0 < \dim N < \dim M$. Then $N$ is a compact Riemann surface and
   \[2\deg E = (m-1) \deg(K_N).\]
\end{proposition}

\begin{proof}
   Let ${\mathbb V} := R^1f_*\KC$. By \cite{VZ2}, Theorem~1 and Remark~2 we have the inequality of slopes $\mu_{K_N}({\mathbb V}) = 2\mu_{K_N}(E) \le \mu_{K_N}(\Omega_N^1)$, where for any torsion free sheave ${\mathcal F}$ of positive rank and ample $H\in Pic(N)$ one defines as usual $\mu_H({\mathcal F}) = \frac{c_1({\mathcal F}).H^{n-1}}{{\rm rk}{\mathcal F}}$. By \Cor{ChernEq}
  \[2\mu_{K_N}(E) = \frac{2}{n+1}c_1(K_N)^n \le \mu_{K_N}(\Omega_N^1) = \frac{1}{n}c_1(K_N)^n.\]
Since $c_1(K_N)^n > 0$ we find $n = 1$. Then $N$ is Riemann surface and we have in fact equality $2\mu_{K_N}(E) = \mu_{K_N}(\Omega_N^1)$, i.e., the Arakelov inequality is an equality.
\end{proof}

\begin{rem} \label{VHS}
Let $f: M \lra N$ be as above, $N$ a compact Riemann surface. Think of $N$ as a quotient of $\Sieg_1$ and of $\pi_1(N)$ as a subgroup of $PSl_2(\KR)$.

As an intermediate result Viehweg and Zuo show  that there exists a lift of $\pi_1(N)$ to a subgroup $\Gamma \subset Sl_2(\KR)$, s.t. $R^1f_*\KC$ comes from the tensor product of the canonical representation of $\Gamma$ and some unitary representation of $\Gamma$ (\cite{VZ}, 1.4. Proposition, \S 2). The fact that $\Gamma$ comes from a quaternion algebra is then a consequence of a result of Takeuchi (\cite{Ta}).

The tensor description of $R^1f_*\KC$ will play a role in the next section. Note that it implies, up to {\'e}tale base change, $E \simeq U \otimes \theta$, where $U$ is a unitary flat bundle and $\theta$ is some theta characteristic on $N$ (compare remark~\ref{remat}).
\end{rem}

\begin{proof}[proof of \Theo{main}]
  Let $M$ be a projective manifold with a h.n.p.c. $\Pi$. which is not K\"ahler Einstein. Then $M \not\simeq \PN_m(\KC)$ and $K_M$ is nef by \Prop{ratcurve}.

Assume $K_M$ abundant and $\Pi$ flat. By \Theo{AbSc} we may assume that the Iitaka fibration $f: M \lra N$ is a smooth abelian group scheme (after some finite {\'e}tale covering). By \Prop{ballq} $N$ is a ball quotient. Since $M$ is neither a torus nor a ball quotient, $0 < \dim N < \dim M$. By \Prop{VZklass}, $N$ is a compact Riemann surface. The equivalence of 1.) and 2.) in \Theo{main} for such abelian group schemes is \cite{VZ}, Theorem 0.5.

It remains to prove that the examples as in 2.) indeed admit a flat projective connection. This will be done in the next section. 
\end{proof}

\begin{rem}
  Let $f: M_m \lra C$ be a an abelian group scheme as in \Theo{main}, 1.), .i.e.,  $2\deg E = (m-1)\deg(K_C) = (m-1)(2g_C-2)$. In the push forward of
  \[0 \lra f^*K_C \lra \Omega_M^1 \lra \Omega_{M/C}^1 \lra 0\]
we have $f_* \Omega_{M/C}^1 \simeq K_C \otimes R^1f_*\sO_M$, implying $H^0(M, f^*K_C) \simeq H^0(M, \Omega_M^1)$. Then $Alb(M) \sim Jac(C)$. The Iitaka fibration essentially coincides with the fibration induced from the albanese map. This might be an idea of how to prove the abundance conjecture in the present case.
\end{rem}

\section{Projective non--K\"ahler--Einstein examples} \label{Ex} \setcounter{equation}{0}
The explicit examples are well known families of abelian varieties (\cite{Shim}, \cite{Mum}, \cite{VZ}, \cite{vG}). They have been studied from many points of view but apparently not as examples of manifolds carrying a flat projective structure.

As in \Theo{main}, let $A$ be a division quaternion algebra defined over some totally real number field $F$ of degree $[F:\KQ] = d$. Assume that $A$ splits at exactly one infinite place, i.e.,
  \begin{equation} \label{decom}
     A \otimes_{\KQ} \KR \simeq M_2(\KR) \oplus \OH \oplus \cdots \oplus \OH.
  \end{equation}
The existence of $A$ follows from Hilbert's reciprocity law.
Let $Cor_{F/\KQ}(A)$ be the rational corestriction of $A$. Then
  \[Cor_{F/\KQ}(A) = M_{2^{d-1}}(B),\]
where $B$ is a quaternionen algebra over $\KQ$ (possibly split). From this data we construct in section~\ref{Alg}
  \begin{enumerate}
   \item a torsion free discrete subgroup $\Gamma$ of $Sl_2(\KR)$ acting canonically on $U_{\KR} = \KR^2$ and on $\Sieg_1$ as a Fuchsian group of first kind,
   \item an orthogonal representation $\rho: \Gamma \lra O(g)$ on $W_{\KR} \simeq \KR^g$ (where $g = 2^{d-1}$ in the case $B$ split and $g = 2^d$ in the case $B$ non split), s.t.
   \item the symplectic representation $id \otimes \rho$ of $\Gamma$ fixes some complete lattice in $U_{\KR} \otimes W_{\KR}$.
  \end{enumerate}
See also remark~\ref{VHS}. We will first explain how the above data leads to an abelian group scheme $Z_\Gamma \lra C_\Gamma = \Sieg_1/\Gamma$ with a projective structure. Here $Z_\Gamma$ will be compact if and only if $C_\Gamma$ is which is the case if the above data is indeed derived from a division quaternion algebra.

The following particular example is not one we are interested in since it is not compact. It is however well known and gives the general idea:

\begin{example}\label{ellcurv} Let $\Gamma \subset Sl_2(\KZ)$ be some standard torsion free congruence subgroup. Consider the quotient $Z$ of $\KC \times \Sieg_1$ by the action
  \[(z, \tau) \mapsto \left(\frac{z+m\tau+n}{c\tau+d}, \frac{a\tau+b}{c\tau+d}\right), \quad \left(\begin{array}{cc}
                      a & b \\
                      c & d
                     \end{array}\right) \in \Gamma, \;\; m,n \in \KZ.\]
Then $Z$ is a smooth manifold with a map $Z \lra C = \Sieg_1/\Gamma$, a modular family of elliptic curve. By example~\ref{Exmpl}, 2.), $Z$ has a projective structure.
\end{example}

\subsection{Construction of the abelian scheme $Z \lra C$}
\subsubsection{Choice of bases.} Point 1.) includes a choice of a basis of $U_{\KR} \simeq \KR^2$. Choose any basis of $W_{\KR}$. Write the elements of $U_{\KR}$ and $W_{\KR}$ as horizontal vectors. The action of $\Gamma$ on $U_{\KR}$ is by left multiplication. The standard symplectic form on $U_{\KR}$
\[\langle u, u'\rangle = u^tJ_2u', \quad J_2 := \left(\begin{array}{cc}
                                           0 & 1 \\
                                          -1 & 0
                                         \end{array}\right)\]
identifies $Sl_2(\KR) = Sp_2(\KR)$ and $U_{\KR} \simeq U^*_{\KR}$ as $\Gamma$--modules. Write the elements of $U^*_{\KR}$ with dual base as horizontal vectors. The action of $\Gamma$ is then given by right multiplication $\gamma(u) = u\gamma^{-1}$, the $\Gamma$--isomorphism $U_{\KR} \simeq U_{\KR}^*$ by $u \mapsto u^tJ_2$.

The choice of a basis for $U_{\KR}$ and $W_{\KR}$ gives an isomorphism
  \[W_\KR \otimes U_{\KR}^* \simeq M_{g \times 2}(\KR).\]
Fix this isomorphism and think of the elements of $W_\KR \otimes U_{\KR}^*$ as real $g \times 2$ matrices from now on. Any $\gamma \in \Gamma$ acts by
   \[\gamma(\alpha) = \rho(\gamma)\alpha\gamma^{-1}, \quad \alpha \in M_{g\times 2}(\KR).\]
Since $W_\KR \otimes U_{\KR}^* \simeq W_\KR \otimes U_{\KR}$ as $\Gamma$--modules, we find a complete lattice $\Lambda \subset M_{g \times 2}(\KR)$ invariant under the action of $\Gamma$ by 3.). By 3.) we find a $\rho(\Gamma)$ invariant symmetric and positiv definit $S \in M_g(\KR)$ s.t. the symplectic form on $M_{g\times 2}(\KR)$ given by
  \begin{equation} \label{E} 
     E(\alpha, \beta) := tr(\alpha^tS\beta J_2)
  \end{equation} 
takes only integral values on $\Lambda$. For later considerations note that $M_{g\times 2}(\KR) \simeq W_\KR \otimes U_{\KR}^*$ is a symplectic $O(S) \times Sl_2(\KR)$ module via $(\delta, \gamma)(\alpha) = \delta\alpha\gamma^{-1}$.

\subsubsection{Definition of $Z$.} Let $\Gamma_{\Lambda}$ be set of matrices
 \[\gamma_{\lambda} := \left(\begin{array}{cc}
      \rho_2(\gamma) & \rho(\gamma)\lambda \\
        0_{2 \times 2} & \gamma
    \end{array}\right) \in Sl_{g+2}(\KR), \quad \gamma \in \Gamma, \lambda \in \Lambda.\]
Then $\Gamma_{\Lambda} \simeq \Lambda \rtimes \Gamma$ is a subgroup of $Sl_{g+2}(\KR)$ and there is an exact sequence
   \[0 \lra \Lambda \lra \Gamma_{\Lambda} \lra \Gamma \lra 1,\]
given by $\lambda \mapsto id_{\lambda}$ and $\gamma_{\lambda} \mapsto \gamma$, respectively.  The projective action on $\KC^g \times \Sieg_1$, i.e., where $\gamma_{\lambda}\in \Gamma_{\Lambda}$ acts by
  \[(z, \tau) \mapsto \left(\frac{\rho_2(\gamma)(z + \lambda {\tau \choose 1})}{c\tau + d}, \frac{a\tau + b}{c\tau + d}\right), \quad \gamma = \left(\begin{array}{cc}
                           a & b \\
                           c & d
                       \end{array}\right),\]
is properly discontinously and free, since the action of $\Gamma$ is. The quotient is a smooth complex manifold 
  \[Z = Z_{\Gamma, \Lambda} := \KC^g \times \Sieg_1/\Gamma_\Lambda.\]
By example~\ref{Exmpl}, 2.), $Z$ has a projective structure. There is a natural holomorphic proper submersion
  \[f: Z \lra C_{\Gamma} := \Sieg_1/\Gamma\]
with a section given by $[\tau] \mapsto [(0,\tau)]$. The fiber $Z_{\tau} = f^{-1}([\tau])$ is isomorphic to $\KC^g$ divided by the corresponding stabilizer subgroup of $\Gamma_{\Lambda}$. Since the action of $\Gamma$ on $\Sieg_1$ is free, 
  \[Z_{\tau} \simeq \KC^g/\Lambda_{\tau},\]
where $\Lambda_{\tau}$ is the image of $\Lambda$ under
  \begin{equation} \label{ComplStr}
    M_{g\times 2}(\KR) \simeq W_{\KR} \times U_{\KR}^* \lra \KC^g, \quad \alpha \mapsto \alpha_{\tau} := \alpha\left(\begin{array}{c}\tau \\ 1 \end{array}\right).
  \end{equation}
\subsubsection{Projectivity of $Z_{\tau}$.} \Formel{ComplStr} endows $M_{g\times 2}(\KR) \simeq W_\KR \otimes U_{\KR}^*$ with the complex structure given by
  \[J_{\tau} = \frac{1}{\Im m \tau}\left(\begin{array}{cc}
                           -\Re e \tau & \tau\bar{\tau} \\
                            -1 & \Re e(\tau)
                        \end{array}\right) \in Sl_2(\KR), \;\;
\mbox{ i.e., }\,\, i \cdot \alpha_{\tau} = (\alpha J_{\tau}^{-1})_{\tau}.\]
If $\tau' = \gamma(\tau)$ for some $\gamma \in Sl_2(\KR)$, then $J_{\tau'} = \gamma J_{\tau}\gamma^{-1}$. Hence $Z_\tau$ is projective if $(M_{g\times 2}(\KR) \simeq W_{\KR} \otimes U_{\KR}^*, \Lambda, J_{\tau}, E)$ satisfies the Riemann conditions:

Recall that $M_{g\times 2}(\KR) \simeq W_{\KR} \otimes U_{\KR}^*$ is a symplectic $O(S) \times Sl_2(\KR)$--module. The complex structure is given by $(id, J_\tau) \in O(S) \times Sl_2(\KR)$. Therefore $E(\alpha J_{\tau}^{-1}, \alpha'J_{\tau}^{-1}) = E(\alpha, \alpha')$ for any $\tau \in \Sieg_1$. Positivity for $\tau = i$ follows from $J_i = J_2$ and $E(\alpha, \alpha J_2^{-1}) = tr(\alpha^tS\alpha) > 0$ for $\alpha \not= 0$. Any $\tau \in \Sieg_1$ is of the form $\tau = \gamma(i)$ for some $\gamma \in Sl_2(\KR)$. The invariance of $E$ and $J_{\tau} = \gamma J_{i}\gamma^{-1}$ shows that $E(\alpha, \alpha' J_{\tau}^{-1})$ is a positive definite symmetric. Therefore, $Z_{\tau}$ is projective.

\subsubsection{Isomorphic $Z_{\tau}$'s.} If $\tau' = \gamma(\tau)$ for some $\gamma \in \Gamma$, then $(Z_{\tau}, E) \simeq (Z_{\tau'}, E)$, where $\varphi: Z_{\tau} \lra Z_{\tau'}$, as a map $\KC^g \lra \KC^g$, is given by
  \begin{equation} \label{analyt}
   \frac{1}{c\tau + d}\rho(\gamma)
  \end{equation}
Indeed,
  \[\Lambda_{\tau'} = \Lambda\left(\begin{array}{c}\gamma(\tau) \\ 1 \end{array}\right) = \frac{1}{c\tau+d}(\Lambda\gamma)_\tau = \frac{\rho(\gamma)}{c\tau+d}(\rho(\gamma^{-1})\Lambda\gamma)_\tau = \frac{\rho(\gamma)}{c\tau+d}\Lambda_{\tau}.\]
On the underlying real vector space $M_{g\times 2}(\KR) \simeq W_\KR \otimes U_\KR$, $\varphi$ is given by $\alpha \mapsto \rho(\gamma)\alpha\gamma^{-1}$ showing that $\varphi$ respects the polarization $E$.

The family of $Z_{\tau}'s$ over $\Sieg_1$, i.e., the quotient $\KC^g \times \Sieg_1/\Lambda$, is projective. We obtain $Z$ by dividing out the action of $\Gamma_{\Lambda}/\Lambda \simeq \Gamma$. Fiberwise this is nothing but \Formel{analyt} proving that the polarizations glue to a section of $R^2f_*\KZ$.

\subsubsection{Map to the moduli space.} This follows from the above considerations, we include it here for convenience of the reader. General reference is \cite{BiLa}:

\subsubsection{Moduli of type $\Delta$ ploarized abelian varieties} Let $\Sieg_g = \{\Pi \in M_g(\KC) | \Pi = \Pi^t, \Im m \Pi > 0\}$. Let $\Delta$ be a {\em type}, i.e., an invertible $g \times g$ diagonal matrix $(\delta_1, \dots, \delta_g)$, $\delta_i \in \KN$ and $\delta_i | \delta_{i+1}$. Fix a basis of $\KR^{2g}$, write the elements as horizontal vectors in the form $(x,y)$, $x, y \in \KR^g$. Choose the letters $(m, n)$ for integral entries. Let
  \[\Lambda_{\Delta} := \{(m, n\Delta) | m, n \in \KZ^g\} \subset \KR^{2g}\]
Any $\Pi \in \Sieg_g$ defines a complex structure on $\KR^{2g}$ via $j_{\Pi}: (x,y) \mapsto x\Pi + y$. As above we have the abelian variety $\KC^g/j_{\Pi}(\Lambda_\Delta)$, where we write the elements of $\KC^g$ as horizontal vectors. Let 
  \[\tilde{\Gamma} = \{\gamma \in Sp_{2g}(\KQ) | \Lambda_{\Delta}\gamma = \Lambda_{\Delta}\}.\]
Then $\tilde{\Gamma}$ acts on $\Sieg_g$ in the usual way. Let $\tilde{\Gamma}_{\Delta}$ be the group consisiting of all maps
  \[\gamma_{\lambda_\Delta} = \left(\begin{array}{ccc}
                              1 & m & n\Delta \\
                              0 & A & B \\
                              0 & C & D
                           \end{array}\right), \;\;  \gamma = \left(\begin{array}{cc}
                              A & B \\
                              C & D
                           \end{array}\right) \in \tilde{\Gamma}, \lambda_\Delta = (m, n\Delta) \in \Lambda_\Delta.\]
As above we have $0 \lra \Lambda_{\Delta} \lra \tilde{\Gamma}_{\Lambda_\Delta} \lra \tilde{\Gamma} \lra 1$ and $\tilde{\Gamma}_{\Lambda_\Delta} \simeq \Lambda_{\Delta} \rtimes \tilde{\Gamma}$. The group $\tilde{\Gamma}_{\Lambda_\Delta}$ acts on $\KC^g \times \Sieg_g$ in 'Grassmanian manner', i.e.,
 \[(z, \Pi) \mapsto \left((z + m\Pi + n\Delta)(C\Pi + D)^{-1}, (A\Pi + B)(C\Pi + D)^{-1}\right).\]
The quotient need not be a manifold since the action of $\tilde{\Gamma}$ need not be free. After replacing $\tilde{\Gamma}$ by some appropriate congruence subgroup, the quotient yields a smooth abelian fibration $U_g \lra A_g$. A relatively ample line bundle is for example induced by the factor of automorphy
  \[a(\gamma_{\lambda_{\Delta}}, (z, \Pi)) = \exp(2i\pi(m\Pi m^t + 2zm^t + (z+m\Pi+n\Delta)(C\Pi + D)^{-1}C(z+m\Pi + n\Delta)^t).\]

\subsubsection{ An equivariant holomorphic map $\KC^g \times \Sieg_1 \lra \KC^g \times \Sieg_g$} is defined as follows: let $\lambda_1, \dots, \lambda_g, \mu_1, \dots, \mu_g \in M_{g \times 2}(\KR)$ be a symplectic lattice basis. It means that these elements generate $\Lambda$ as a $\KZ$--module and that $E$ from \Formel{E} is in this basis given by
  \[\left(\begin{array}{cc}
            0 & \Delta \\
           -\Delta & 0
          \end{array}\right) \in M_{2g}(\KZ),\]
where $\Delta$ is a certain type. Fix the basis $\lambda_1, \dots, \lambda_g, \mu'_1 = \frac{1}{\delta_1}\mu_1, \dots, \mu'_g = \frac{1}{\delta_g}\mu_g$ of $M_{g \times 2}(\KR)$ in which $E$ is given in standard form. The choice defines a map $\kappa: M_{g\times 2}(\KR) \lra \KR^{2g}$. Write the elements of $\KR^{2g}$ again as horizontal vectors. Then $\kappa(\Lambda) = \Lambda_{\Delta}$.

For any $\gamma \in \Gamma$, $\alpha \mapsto \rho(\gamma^{-1})\alpha\gamma$ is a symplectic automorphism of $M_{g\times 2}(\KR)$. We find a matrix $\sigma(\gamma) \in Sp_{2g}(\KQ)$ such that $\kappa(\rho(\gamma^{-1})\alpha\gamma) = \kappa(\alpha)\sigma(\gamma)$. The map $\sigma$ is a group morphism $\Gamma \lra \tilde{\Gamma}$. It extends to 
  \begin{equation} \label{grpm} 
      \Gamma_{\Lambda} \lra \tilde{\Gamma}_{\Lambda_{\Delta}}, \quad \gamma_{\lambda} \mapsto \sigma(\gamma)_{\kappa(\lambda)}.
  \end{equation}
The map $\KC^g \times \Sieg_1 \lra \KC^g \times \Sieg_g$ is finally defined as follows: For any $\tau \in \Sieg_1$, using \Formel{ComplStr}, let
  \[\Pi_1(\tau) = (\lambda_{1, \tau}, \dots, \lambda_{g, \tau}) \in M_{g}(\KC),\]
\[\Pi'_2(\tau) = (\mu'_{1, \tau}, \dots, \mu'_{g, \tau}) \in M_{g}(\KC).\]
Then $(\Pi_1(\tau), \Pi'_2(\tau)\Delta)$ is the period matrix of the abelian variety $Z_{\tau}$ with respect to the chosen bases. By \Formel{analyt} we have $(\Pi_1(\gamma(\tau)), \Pi'_2(\gamma(\tau)) = \frac{\rho(\gamma)}{c\tau+d}(\Pi_1(\tau), \Pi'_2(\tau))\sigma(\gamma)^t$. The matrix $\Pi'_2(\tau)$ is invertible by and the Riemann bilinear relations in this form read
  \[\Pi(\tau) := [\Pi'_2(\tau)]^{-1}\Pi_1(\tau) \in \Sieg_g.\]
Note that $(\Pi(\tau), \Delta)$ is the period matrix of $Z_\tau$ with respect to the symplectic lattice basis and the basis of $\KC^g$ induced by $\Pi'_2(\tau)$. The map $\KC^g \times \Sieg_1 \lra \KC^g \times \Sieg_g$, $(z, \tau) \mapsto (\Pi'_2(\tau)^{-1}z, \Pi(\tau))$ is the desired holomorphic map, equivariant with respect to \Formel{grpm}.

\subsection{From $A$ to $Z$} \label{Alg} It remains to show how a quaternion algebra $A$ as above leads to a collection of data 1) --- 3). We first recall some results

\subsubsection{On central simple algebras.} Let $A$ a central simple algebra of finite dimension over a field $K$. It is called {\em division} if it is a skew field. It is called a {\em quaternion algebra} if $[A:K] = 4$. A quaternion algebra is either division or {\em split}, i.e., $A \simeq M_2(K)$. Let $Br(K)$ be the Brauer group of $K$. The order $e(A)$ of $[A] \in Br(K)$ is finite and is called the {\em exponent} of $A$. A theorem of Wedderburn says $A \simeq M_r(D)$, where $D/K$ is a division algebra. The $K$--dimension $[D:K] = s(D)^2$ for some $s(D)=s(A) \in \KN$ called {\em (Schur-) index} of $A$. One has $e(A) | s(A)$ and if $K$ is a local or global field, then even $e(A) = s(A)$.

The corestriction $Cor_{F/\KQ}(A)$ is a $4^d$ dimensional central simple $\KQ$--algebra. The corestriction induces a map of Brauer groups
  \[Br(F) \lra Br(\KQ).\]
Since $e(A)=s(A) = 2$ we have $e(Cor_{F/\KQ}(A)) = 1$ or $2$. Hence
  \begin{equation} \label{Bdef}
    Cor_{F/\KQ}(A) \simeq M_{2^d}(\KQ) \quad \mbox{or} \quad Cor_{F/\KQ}(A) \simeq M_{2^{d-1}}(B),
  \end{equation}
where $B$ is a division quaternion algebra over $\KQ$. We refer to the first case as ``$B$ splits' because here $Cor_{F/\KQ}(A) \simeq M_{2^{d-1}}(B)$ for $B = M_2(\KQ)$. Consider the following $\KQ$--vector spaces 
 \[V_{\KQ} = \KQ^{2^d}, \mbox{ in the case $B$ split}, \; V_{\KQ} = B^{2^{d-1}}, \mbox{ in the case $B$ non--split.}\]
The $\KQ$--dimension is $2^d$ and $2^{d+1}$, respectively. The elements of $V_{\KQ}$ are cosidered as vertical vectors and $V_{\KQ}$ as a $Cor_{F/\KQ}(A)$--module. 

From $A \otimes_{\KQ} \KR \simeq M_2(\KR) \oplus \OH^{\oplus d-1}$ using $\OH \otimes_{\KR} \OH \simeq M_4(\KR)$ we obtain
 \begin{equation} \label{CorR} Cor_{F/\KQ}(A) \otimes_{\KQ} \KR \simeq M_2(\KR) \otimes_{\KR} \OH^{\otimes_{\KR} d-1} \simeq
 \left\{\begin{array}{ll}
         M_{2^d}(\KR), & d \mbox{ odd} \\
         M_{2^{d-1}}(\OH), & d \mbox{ even.}
      \end{array}\right.
 \end{equation}
Then $B$ is indefinit (i.e., $B \otimes_{\KR} \KR \simeq M_2(\KR)$) if and only if $d$ is odd, $B$ is definit (i.e., $B \otimes_{\KR} \KR \simeq \OH)$) if and only if $d$ is even and
\begin{equation} \label{VR} V_{\KR} := V_{\KQ} \otimes_{\KQ} \KR \simeq 
 \left\{\begin{array}{ll}
           \KR^{2^d}, & d \mbox{ odd and $B$ split} \\
           M_{2^d \times 2}(\KR), & d \mbox{ odd and $B$ non split} \\
           \OH^{2^{d-1}} \simeq \KC^{2^{d}}, & d \mbox{ even}
      \end{array}\right.
 \end{equation}
with the obvious action of $Cor_{F/\KQ}(A) \otimes_{\KQ} \KR$ from \Formel{CorR}. Note that $V_{\KR}$ is an irreducible  $(Cor_{F/\KQ}(A)  \otimes_{\KQ} \KR)^\times$ module in the first and third case, while it is a direct sum of two isomorphic modules in the second.

\

The corestriction comes with a map
  \[Nm: A^{\times} \lra Cor_{F/\KQ}(A)^{\times}.\]
Let $x \mapsto x'$ be the canonical involution of $A$ and
  \[G = \{x \in A | xx' = 1\}.\]
Then $G$ is an algebraic group over $\KQ$. Via $Nm$, $G(\KQ)$ acts on $V_{\KQ}$. Let $\Lambda \subset V_{\KQ}$ be some complete lattice, $\Gamma \subset G(\KQ)$ be some torsion free arithmetic subgroup fixing $\Lambda$ via $Nm$.

Consider the action over $\KR$: The elements of norm $1$ in $M_2(\KR)$ and $\OH$ form the groups $Sl_2(\KR)$ and $SU(2)$, respectively. Then $A \otimes_{\KQ} \KR \simeq M_2(\KR) \oplus \OH^{\oplus d-1}$ shows
  \[G(\KR) = Sl_2(\KR) \times \underbrace{SU(2) \times \cdots \times SU(2)}_{d-1}\]
The isomorphism $\OH \otimes_{\KR} \OH \simeq M_4(\KR)$ induces $SU(2) \times SU(2) \lra SO(4)$. Then $Nm$ factors over
 \begin{equation} G(\KR) \lra \left\{\begin{array}{llc}
         Sl_2(\KR) \times SO(2^{d-1}) , & & \mbox{$B$ indefinit / $d$ odd} \\
         Sl_2(\KR) \times SU(2^{d-1}), &  & \mbox{$B$ definit / $d$ even}.
      \end{array}\right.
 \end{equation}
Projection onto the first factor embeds $\Gamma$ into $Sl_2(\KR)$. Projection onto the second factor induces an orthogonal representation $\rho'$ of $\Gamma$ (in the case $d$ even take the real part of the unitary form). In the case $B$ split or definit put $\rho = \rho'$, in the case $B$ non--split indefinite put $\rho = \rho' \oplus \rho'$ and , the direct sum of. Then $Nm = id \otimes \rho$ over the reals fixing, by construction, a complete latice $\Lambda$. This gives 1) -- 3).

\begin{rem}
1.) The group $G$ can be identified with the special Mumford Tate group of the constructed family of abelian varieties. The ring $End_{\KQ}(Z_{\tau})$ for $\tau$ general is isomorphic to the group of endomorphisms of $V$ commuting with the action of $G$. 

In the case $B$ non split, $V_{\KQ}$ becomes a $B$--module via $\beta v := v\beta'$ and this action clearly commutes. This gives an embedding $B \hookrightarrow End_{\KQ}(Z_{\tau})$. Together with results from (\cite{VZ}, where $B$ is not explicitely mentioned) we find $End_{\KQ}(Z_{\tau}) \simeq B$ in the case $B$ non--split and $End_{\KQ}(Z_{\tau}) \simeq \KQ$ in the case $B$ split.

2.) Assume 1.)--3.) as above are given corresponding to the family $Z \lra C$. In 3.) we may replace $\rho$ by a direct sum of, say, $r$ copies of $\rho$. This again gives 1.)--3.). corresponding to the $r$--fold product $Z \times_C \cdots \times_C Z$. 
\end{rem}

Let us illustrate the construction with one more explicit example.

\begin{example}(False elliptic Curves) \label{FEK}
(\cite{Shim}) This is the case $d=1$, $F = \KQ$, $A = B$ indefinit. Denote quaternion conjugation in $B$ by $\alpha \mapsto \alpha'$. Choose some pure quaternion $y$ (i.e., $y = -y'$) such that $b := y^2 < 0$. There exists $x \in B$ such that $xy=-yx$ and $a := x^2 > 0$. Then $B$ is generated by $x$ and $y$ as an algebra over $\KQ$ and
  \[x \mapsto \left(\begin{array}{cc}
                              \sqrt{a} & 0 \\
                                 0 & -\sqrt{a}
                            \end{array}\right), \quad y \mapsto \left(\begin{array}{cc}
 0 & b \\
 1 & 0
\end{array}\right)\]
gives an emedding $B \hookrightarrow M_{2\times 2}(\KR)$.  Identify $B$ with its image in $M_{2\times 2}(\KR)$, s.t. reduced norm and trace are then given by usual matrix determinant and trace, respectively.

Let $\Lambda$ be some complete lattice in $V_\KQ := B$. For the construction we may take a maximal order for $\Lambda$. Let $\Gamma$ be a torsion free subgroup of the unit subgroup $\Lambda^\times_1 = \{\alpha \in \Lambda | \alpha\alpha'=1\}$. The matrix $S := J_2y$ is symmetric and positive definite. Let $\rho: \Gamma \lra O(\KR^2, S)$ be the trivial representation. Then all of the above points are satisfied. 

Indeed, $\Gamma$ acts on $V_\KR = M_{2}(\KR)$ via $\alpha \mapsto \rho(\gamma)\alpha\gamma^{-1} = \alpha\gamma^{-1}$, fixing $\Lambda$. The form $E(\alpha, \beta) = tr(\alpha^tS\beta J_2)$ from \Formel{E} is $\Gamma$--stable. The extension of quaternion conjugation from $B$ to $V_\KR = M_{2}(\KR)$ is given by $\alpha \mapsto J_2^{-1}\alpha^tJ_2$. Then (recall $tr(\alpha\beta) = tr(\beta\alpha)$)
  \[E(\alpha, \beta) = tr(\alpha^tS\beta J_2) = tr(-J_2\beta^tS\alpha) = tr(-J_2y\alpha J_2\beta^t) =\]\[=  tr(y\alpha J_2^{-1}\beta^tJ_2) = tr(y \alpha\beta').\]
The last description shows that $E$ only takes rational values on $B$. Then some multiple of $E$ only takes integral values on $\Lambda$ ($E$ is Shimura's form from \cite{Shim}).
\end{example}

\end{document}